\documentclass[11pt, a4paper]{amsart}
\usepackage{amssymb}
\usepackage{amsmath}
\usepackage{amssymb}
\usepackage{amsthm}
\usepackage{bbm}
\usepackage{bm}
\usepackage{mathtools}
\usepackage{mathrsfs}
\usepackage[T1]{fontenc}
\usepackage[usenames,dvipsnames,svgnames,table]{xcolor}
\usepackage{esint}
\allowdisplaybreaks

\usepackage{enumitem}
\setenumerate{label={\rm (\alph{*})}}
\usepackage[colorlinks=true,citecolor=black,urlcolor=black,
linkcolor=black]{hyperref}

\usepackage{graphicx}
\usepackage{tikz}

\usepackage{geometry}
\makeatletter
\newcommand*\bigcdot{\mathpalette\bigcdot@{.5}}
\newcommand*\bigcdot@[2]{\mathbin{\vcenter{\hbox{\scalebox{#2}{$\m@th#1\bullet$}}}}}
\makeatother

\newtheorem{theorem}{Theorem}
\newtheorem{lemma}[theorem]{Lemma}
\newtheorem{proposition}[theorem]{Proposition}

\newtheorem{remark}[theorem]{Remark}
\newtheorem{example}[theorem]{Example}

\usepackage{booktabs}

\normalsize

\def\XXint#1#2#3{{\setbox0=\hbox{$#1{#2#3}{\int}$ }
\vcenter{\hbox{$#2#3$ }}\kern-.6\wd0}}

\newcommand{\bv}{\operatorname{BV}}

\newcommand{\di}{\operatorname{div}}

\newcommand{\dif}{\operatorname{d}\!}

\newcommand{\R}{\mathbb{R}}
\newcommand{\A}{\mathbb{B}}

\newcommand{\locc}{\operatorname{loc}}

\newcommand{\sobo}{\operatorname{W}}

\newcommand{\lebe}{\operatorname{L}}

\newcommand{\hold}{\operatorname{C}}

\newcommand{\D}{D}
\newcommand{\curl}{\operatorname{curl}}
\renewcommand{\leq}{\leqslant}

\newcommand{\imag}{\operatorname{i}}
\newcommand{\id}{\operatorname{Id}}

\newcommand{\e}{\operatorname{e}}

\renewcommand{\di}{\operatorname{div}}
\newcommand{\bmo}{\operatorname{BMO}}

\newcommand{\lin}{\operatorname{Lin}}

\renewcommand{\e}{\operatorname{e}}

\newcommand{\rank}{\operatorname{rank}}

\usepackage{tikz-cd}
\AtEndDocument{\bigskip{\footnotesize%
  \textsc{Andrew Wiles Building, University of Oxford, Woodstock Rd, Oxford OX2 6GG.} \par  
  \noindent\textit{E-mail address}, B.~Rai\c{t}\u{a}: \texttt{raita@maths.ox.ac.uk}
}}

\begin{document}
\title[$\lebe^1$--estimates for constant rank operators]{$\lebe^1$--estimates for constant rank operators}
\author[B. Rai\c{t}\u{a}]{Bogdan Rai\c{t}\u{a}}
\begin{abstract}
We show that the inequality
\begin{align*}
\|D^{k-1}(u-\pi u)\|_{\lebe^{n/(n-1)}}\leq c\|\A(D) u\|_{\lebe^1}
\end{align*}
holds for vector fields $u\in\hold^\infty_c$ if and only if $\A$ is canceling. Here $\pi$ denotes the $\lebe^2$--orthogonal projection onto the kernel of the $k$--homogeneous differential operator $\A(D)$ of \emph{constant rank}. Other critical embeddings are established. 
\end{abstract}
\maketitle
Recently, \textsc{Van Schaftingen} proved in \cite{VS} that the surprising estimate
\begin{align}\label{eq:VS}
\|D^{k-1}u\|_{\lebe^{n/(n-1)}}\leq c\|\A(D) u\|_{\lebe^1}
\end{align}
for vector fields $u\in\hold^{\infty}_c$ holds if and only if the $k$--homogeneous differential operator $\A(D)$ with constant coefficients is elliptic and canceling. Here (overdetermined) ellipticity is defined by injectivity of the symbol map (Fourier multiplier) of $\A(D)$, whereas cancellation, for elliptic operators, can be defined as the property that the equation $\A(D)u=\delta_0w$ has no distributional solutions if $w\neq0$.

The result of \textsc{Van Schaftingen} can be viewed as a critical a priori estimate for linear elliptic equations on the Sobolev scale. This is in sharp contrast with the case $1<p<\infty$ of $\lebe^p$--estimates,
\begin{align}\label{eq:CZ}
\|D^ku\|_{\lebe^p}\leq c\|\A(D)u\|_{\lebe^p}
\end{align}
for $u\in\hold^\infty_c$, which is attributable to \textsc{Calder\'on} and \textsc{Zygmund} \cite{CZ}. On the other hand, the estimate~\eqref{eq:CZ} holds for $p=1$ only if $D^ku=L[\A(D)u]$ for all $u$, where $L$ is a linear map between finite dimensional spaces. In other words, only trivial estimates such as \eqref{eq:CZ} hold for $p=1$ \cite{Ornstein,KK}. 

The purpose of the present note is to relax the ellipticity assumption to the constant rank condition, which gives the largest class in which existent methods apply (see Remark~\ref{rk:non-CR}). Only relatively few examples outside this class are known and it is safe to say that developing a theory in this regime is a challenging problem (see Remark~\ref{rk:non-CR}). Of course, ellipticity is necessary for \eqref{eq:VS}, as is for \eqref{eq:CZ}, but as noted in \cite{SW} the inequality \eqref{eq:CZ} holds for $p=2$ and operators of constant rank if fewer competitor maps $u\in\hold^\infty_c$ are allowed (to be precise, if $u\in[\ker\A(D)]^\perp$). We will show that
\begin{align}\label{eq:BR}
\|D^{k-1}(u-\pi u)\|_{\lebe^{n/(n-1)}}\leq c\|\A(D) u\|_{\lebe^1}
\end{align}
holds for vector fields $u\in\hold^\infty_c$ if and only if $\A(D)$ is canceling. Here $\pi$ denotes the $\lebe^2$--projection onto $\ker\A$ (the complete statement as well as explicit definitions are given in Section~\ref{sec:statements}).

The proof follows from the fact that $u-\pi u$ can be written as a Fourier multiplier of $\A(D)u$ and of the dual estimate \cite[Thm.~1.4]{VS}. To apply the latter, we employ the recent construction of exact annihilators of constant rank operators \cite{R}. We do not claim that the ideas used are original, but we make a case of the fact that a considerable ammount of work has led to, among other developments, the characterization of operators that satisfy \eqref{eq:VS} \cite{BBM1,BBM2,BBM3,BB02,BBCR,BB04,BB07,VS-1,VS0,VS1,VS4} and to round off the theory with the case of constant rank operators seems welcome.

With similar methods, we obtain critical inequalities on the full Sobolev scale, fractional scales, Lorentz scale, Hardy and Hardy--Sobolev inequalities.
\section{The inequalities}\label{sec:statements}
\subsection{Definitions and notation}\label{sec:def}
Throughout this note we will work with a $k$--homogeneous linear differential operator with constant coefficients on $\R^n$ from $V$ to $W$,
i.e.,
\begin{align}\label{eq:A}
\A(D)u\coloneqq\sum_{|\alpha|=k}\A_\alpha \partial^\alpha u\qquad\text{ for }u\in\hold^\infty(\R^n,V),
\end{align}
where $\A_\alpha\in\lin(V,W)$ are linear maps for all multi--indices $\alpha$ with $|\alpha|=k$ and $V,\,W$ are finite dimensional inner product spaces. The symbol map of $\A(D)$ is a $\lin(V,W)$--valued homogeneous polynomial of degree $k$ given by
\begin{align*}
\A(\xi)=\sum_{|\alpha|=k}\xi^\alpha\A_\alpha\in\lin(V,W)\qquad\text{ for }\xi\in\R^n.
\end{align*}
The symbol map is defined such that $\widehat{\A(D)u}=(-\imag)^k\A(\cdot)\hat{u}$ for $u$ in the Schwartz space $\mathscr{S}(\R^n,V)$. Our convention for the Fourier transform is
\begin{align*}
\hat{u}(\xi)=\int_{\R^n}\e^{-\imag \xi\cdot x}u(x)\dif x.
\end{align*}

An operator $\A(D)$ is said to be of \emph{constant rank} if for some integer $r$,
\begin{align*}
\rank\A(\xi)=r\qquad\text{ for all }0\neq\xi \in\R^n.
\end{align*}
The constant rank condition was introduced in \cite{SW} to prove coerciveness estimates for non--elliptic systems and used in the theories of compensated compactness \cite{Murat} and lower semi--continuity \cite{FM}. It is well--known that the constant rank condition is equivalent with smoothness of the $(-k)$--homogenous map $\xi\mapsto\A^\dagger(\xi)$ away from zero (here $M^\dagger$ denotes the \emph{Moore--Penrose generalized inverse} of a matrix $M$; see \cite{PI} for a general presentation). This fact was recently improved in \cite{R}, where it is shown, as an immediate consequence of the result in \cite{Decell}, that $\A^{\dagger}(\cdot)\in\hold^\infty(\R^n\setminus\{0\},\lin(W,V))$ is a rational function.

An operator $\A(D)$ is said to be (overdetermined) \emph{elliptic} if $\A(D)$ is of full constant rank ($r=\dim V$); equivalently, the symbol map $\A(\xi)$ is injective for all $0\neq\xi\in\R^n$. An operator $\A(D)$ is said to be \emph{canceling} if
\begin{align*}
\bigcap_{\xi\in\mathbb{S}^{n-1}}	\mathrm{im\,}\A(\xi)=\{0\}.
\end{align*}
It follows from various considerations in \cite{VS} that, if $\A(D)$ is elliptic, then it is canceling if and only if $\A(D)u=\delta_0w$ implies $w=0$. The same holds true for constant rank operators:
\begin{lemma}\label{lem:canc}
Let $\A(D)$ as in \eqref{eq:A} have constant rank. Then $\A(D)$ is canceling if and only if for $u\in\lebe^1_{\locc}$ and $w\in W$,
\begin{align}\label{eq:canc_eq}
\A(D) u=\delta_0w\implies w=0.
\end{align}
Moreover, if $\A(D)$ is non--canceling, \eqref{eq:canc_eq} has a solution $u_h\in\hold^\infty(\R^n\setminus\{0\},V)$ such that $\D^lu_h$ is $(k-n-l)$--homogeneous for all $l\geq\max\{0,k-n+1\}$.
\end{lemma}
This fact follows just as in \cite[Lem.~2.5]{Rdiff}, with the modification that $\mathcal{A}$ in the respective proof should be replaced with $\mathcal{A}$ as constructed in \cite[Rk.~3]{R}. We will return to this point in detail later.

Schematically, the strategy in \cite{VS} for elliptic and canceling $\A(D)$ can be summarized as follows:
\begin{enumerate}
\item\label{it:step1} By ellipticity, write $u=K\star (\A(D)u)$ where $\hat{K}=\A^\dagger(\cdot)$ \cite[Lem.~2.1]{BVS};
\item\label{it:step2} Use an elliptic estimate (in principle, boundedness of singular integral operators) to obtain
\begin{align*}
\|u\|_X\leq c\|\A(D) u\|_Y \qquad\text{ for }u\in\hold^\infty_c(\R^n,V).
\end{align*}
For example, in the case of \eqref{eq:VS},  $X=\dot{\sobo}{^{k-1,n/(n-1)}}$ and $Y=\dot{\sobo}{^{-1,n/(n-1)}}$.
\item\label{it:step3} Use a dual, critical $\lebe^1$--estimate
\begin{align*}
\|w\|_{Y}\leq c\|w \|_{\lebe^1}
\end{align*}
for vector fields $w\in\hold^\infty_c(\R^n,W)$ in the kernel of an underdetermined differential operator, $\mathcal{A}(D)w=0$ (see \cite[Thm.~1.4]{VS}).
\end{enumerate}
The operators that satisfy \ref{it:step3} are said to be \emph{cocanceling} and are defined by
\begin{align*}
\bigcap_{\xi\in\mathbb{S}^{n-1}}	\ker\mathcal{A}(\xi)=\{0\}.
\end{align*}
For applicability of \ref{it:step3} to improve the estimate in \ref{it:step2} with $w=\A u$, we want
\begin{align}\label{eq:exact}
\mathrm{im\,}\A(\xi)=\ker\mathcal{A}(\xi)\qquad\text{ for all }0\neq\xi\in\R^n,
\end{align}
in which case we say that the differential operator $\mathcal{A}(D)$ is an exact annihilator of $\A(D)$. Given a differential operator $\A(D)$, a naive way to construct $\mathcal{A}(D)$ would be to set
\begin{align*}
\mathcal{A}(\xi)=\mathrm{Proj}_{[\mathrm{im\,}\A(\xi)]^\perp}=\id_W-\A(\xi)\A^{\dagger}(\xi)\eqqcolon \mathcal{P}(\xi),
\end{align*}
but this expression does not define a differential operator in general. However, if $\A(D)$ is elliptic, one can explicitly compute that $\A^\dagger(\xi)=[\A^*(\xi)\A(\xi)]^{-1}\A^*(\xi)$, so that setting
\begin{align*}
\mathcal{A}(\xi)\coloneqq \det[\A^*(\xi)\A(\xi)]\mathcal{P}(\xi)
\end{align*}
defines a polynomial, and hence $\mathcal{A}(D)$ thus defined is a homogeneous linear differential operator \cite[Sec.~4.2]{VS}.

If we only assume that $\A(D)$ has constant rank, one still has that $\A^\dagger(\xi)=p(\xi)^{-1}\mathcal{Q}(\xi)$, where $p$ is a scalar--valued homogeneous polynomial with no non--zero roots in $\R^n$ and $\mathcal Q$ is a $\lin(W,V)$--valued polynomial. In particular,
\begin{align*}
\mathcal{A}(\xi)\coloneqq p(\xi)\mathcal{P}(\xi)
\end{align*}
defines a linear differential operator that satisfies \eqref{eq:exact}. Collecting, we proved that:
\begin{lemma}[{\cite[Rk.~4]{R}}]\label{lem:annihilator}
Let $\A(D)$ be a constant rank operator as in \eqref{eq:A}. Then there exists a constant rank, homogeneous, linear differential operator $\mathcal{A}(D)$ with constant coefficients such that the exact relation \eqref{eq:exact} holds.
\end{lemma}
Lemma~\ref{lem:annihilator} deals with the duality of estimates \ref{it:step2} and \ref{it:step3} (see also Remark~\ref{rk:duality}). On the other hand, when the ellipticity assumption in \ref{it:step1} is relaxed to the constant rank condition, the convolution $K\star(\A(D)u)$ is no longer a representation of a Schwartz function $u$. Instead, we can compute in frequency space:
\begin{align*}
\A^{\dagger}(\xi)\widehat{\A(D)u}(\xi)&=\A^\dagger(\xi)\A(\xi)\hat{u}(\xi)=\hat{u}(\xi)-\left[\id_V-\A^\dagger(\xi)\A(\xi)\right]\hat{u}(\xi)\\
&=\hat{u}(\xi)-\mathrm{Proj}_{\ker\A(\xi)}\hat{u}(\xi),
\end{align*}
so that we actually have
\begin{align}\label{eq:defs}
u-\pi u=K\star (\A(D)u)\quad\text{ where }\quad\hat{K}=\A^\dagger(\cdot)\quad\text{ and }\quad\hat{\pi}=
\id_V-\A^\dagger(\cdot)\A(\cdot).
\end{align}
In fact, it is easy to see that, since $\xi\mapsto \id_V-\A^\dagger(\xi)\A(\xi)$ is zero--homogeneous, $\pi$ maps $\mathscr{S}(\R^n,V)$ into $\mathscr{S}(\R^n,V)$. It can also be checked by direct computation that $\pi$ is the $\lebe^2$--orthogonal projection onto $\ker\A\cap\mathscr{S}$. 

Of course, if $\A(D)$ is elliptic, then $\pi\equiv0$, so all our inequalities for constant rank operators exactly reduce to the known existing inequalities for elliptic operators.

Finally, we refer the reader to \cite{Triebel} for background on Besov and Triebel--Lizorkin spaces and to \cite{Ziemer} for background on Lorentz spaces.
\subsection{The statements}\label{sec:statements(for_real)}
The considerations in Section~\ref{sec:def} are sufficient so that we need not write complete proofs. All necessary details can be found in \cite{VS,BVS,Rdiff} and the references therein. We will indicate which result should be quoted for each statement. We use the notation in \eqref{eq:defs}.
\begin{theorem}[$\lebe^p$--scales]\label{thm:Lp}
Let $\A(D)$ be a constant rank operator as in \eqref{eq:A} and $n>1$. Then the estimate
\begin{align*}
\|u-\pi u\|_{X}\leq c\|\A(D)u\|_{\lebe^1}\qquad\text{ for all }u\in\hold^\infty_c(\R^n,V)
\end{align*}
holds if and only if $\A(D)$ is canceling. Here $X$ is one of the following:
\begin{enumerate}
\item\label{it:Lp-sobo} $X=\dot{\sobo}{^{k-j,n/(n-j)}}$ for some $1\leq j \leq \min\{k,n-1\}$.
\item\label{it:Lp-frac} $X=\dot{\sobo}{^{k-s,n/(n-s)}}$ for some $s\in(0,n)$.
\item\label{it:Lp-beso} $X=\dot{\mathrm{B}}{^{k-s,n/(n-s)}_q}$ for some $s\in(0,n)$, $1< q < \infty$.
\item\label{it:Lp-trli} $X=\dot{\mathrm{F}}{^{k-s,n/(n-s)}_q}$ for some $s\in(0,n)$, $1\leq q\leq \infty$.
\item\label{it:Lp-lore} $X=\dot{\sobo}{^{k-j}\lebe^{n/(n-j),q}}$ for some $1\leq j \leq \min\{k,n-1\}$, $1<q<\infty$.
\end{enumerate}
\end{theorem}
The limiting case $s=0$ is ruled out by Ornstein's Non--inequality. Theorem~\ref{thm:Lp}\ref{it:Lp-sobo} can be refined on a Hardy--Sobolev scale (cp.~\cite[Prop.~2.3]{BVS}):
\begin{theorem}[Hardy--type inequalities]\label{thm:hardy}
Let $\A(D)$ be a constant rank operator as in \eqref{eq:A}, $n>1$, $1\leq j\leq \min\{n-1,k\}$, and $q\in[1,n/(n-j)]$. Then the estimate
\begin{align*}
\left(\int_{\R^n}\left(\frac{|D^{k-j}(u-\pi u)(x)|}{|x|^{n/q-(n-j)}}\right)^q\dif x\right)^{1/q}\leq c\int_{\R^n}|\A(D)u(x)|\dif x\quad\text{ for }u\in\hold^\infty_c(\R^n,V)
\end{align*}
holds if and only if $\A(D)$ is canceling.
\end{theorem}
It was already noticed in \cite{Rdiff} for elliptic operators that, in  the limiting cases $s= n$ of Theorem~\ref{thm:Lp}\ref{it:Lp-frac},~\ref{it:Lp-beso}, $j=n$ of Theorem~\ref{thm:Lp}\ref{it:Lp-sobo} and Theorem~\ref{thm:hardy}   (which only makes sense if $q=\infty$), the canceling condition should be replaced by a strictly weaker condition \cite[Eq.~(1.7)]{Rdiff}. In the present context, we extend \cite[Thm.~1.3]{Rdiff}:
\begin{theorem}[$\lebe^\infty$--estimate]\label{thm:Linfty}
Let $\A(D)$ be a constant rank operator as in \eqref{eq:A} and $k\geq n\geq1$. Then the estimate
\begin{align*}
\|D^{k-n}(u-\pi u)\|_{\lebe^\infty}\leq c\|\A u\|_{\lebe^1}\qquad\text{ for all }u\in\hold^\infty_c(\R^n,V)
\end{align*}
holds if and only if $\A(D)$ satisfies
\begin{align}\label{eq:mist_cond2}
\int_{\mathbb{S}^{n-1}}\A^{\dagger}(\xi)w\otimes^{k-n}\xi\dif\mathscr{H}^{n-1}(\xi)=0\quad\text{ for all }\quad w\in\bigcap_{\xi\in\mathbb{S}^{n-1}}\mathrm{im\,}\A(\xi).
\end{align}
\end{theorem}
We can also interpret Theorem~\ref{thm:Linfty} as the limiting case $j=n$ of Theorem~\ref{thm:Lp}\ref{it:Lp-lore} with $q=1$ (see \cite[Sec.~7]{Rdiff}).

The proof of necessity for all three results stems from the following Lemma:
\begin{lemma}\label{lem:rhoKw}
Let $\A(D)$ be a constant rank operator as in \eqref{eq:A}. Let $w\in W$ be such that $w\in\mathrm{im\,}\A(\xi)$ for all $\xi\neq0$ and $\rho\in\hold^\infty_c(B_2)$ be such that $\rho=1$ in $B_1$. Then the locally integrable function $\rho Kw$ can be approximated by $\hold^\infty_c$--functions in the semi--norm $u\mapsto |\A u|(\R^n)$. Here $|\mu|$ denotes the total variation measure of a measure $\mu$, $K$ is such that $\hat{K}=\A^\dagger(\cdot)$, and $B_r\coloneqq B(0,r)$. Moreover,
\begin{align}\label{eq:conv_legit}
\rho Kw-\pi (\rho Kw)=
Kw+\Psi,
\end{align}
where $\Psi\in\hold^\infty(\R^n,V)$.
\end{lemma}
\begin{proof}
We first recall that $\A^\dagger(\cdot)$ is a $(-k)$--homogeneous distribution in $\R^n\setminus\{0\}$, so it extends to a tempered distribution (see the proof of \cite[Lem.~2.1]{BVS}). Moreover, $K\in\hold^\infty(\R^n\setminus\{0\})\cap\lebe^1_{\locc}(\R^n)$ valued in $\lin(W,V)$ and $\A(D)(Kw)=\delta_0w$\footnote{We can take $u_h=Kw$ in Lemma~\ref{lem:canc}.}, where $w$ is as in the statement of the Lemma. This is due to the fact that
\begin{align*}
\mathscr{F}[{\A(D)(Kw)}](\xi)=\A(\xi)\A^\dagger(\xi)w=w,
\end{align*}
where the last equality follows from the fact that $MM^\dagger$ is the orthogonal projection onto $\mathrm{im\,}M$. By the chain rule and $\rho(0)=1$, we have that
\begin{align*}
\A(D)(\rho Kw)= \delta_0w+\psi,
\end{align*}
where $\psi\in\hold^\infty_c(B_{5/2}\setminus B_{1/2},W)$, so that $\A(D)(\rho Kw)$ is a compactly supported measure.

Let now $\eta_\varepsilon$ be a sequence of radially symmetric, positive (standard) mollifiers. Then $(\rho KW)\star\eta_\varepsilon\in\hold_c^\infty(\R^n,V)$ and 
\begin{align*}
\A(D)[(\rho Kw)\star\eta_\varepsilon]=[\A(D)(\rho K w)]\star\eta_\varepsilon=\eta_\varepsilon w+\psi\star\eta_\varepsilon.
\end{align*}
For small $\varepsilon>0$, the two smooth terms above are disjointly supported, so
\begin{align*}
\int_{\R^n}|\A(D)[(\rho K w)\star\eta_\varepsilon|\dif x&= \int_{B_{1/2}}|\eta_\varepsilon w|\dif x+\int_{B_{5/2}\setminus B_{1/2}}|\psi\star\eta_\varepsilon|\dif x\\
&\rightarrow |w|+\|\psi\|_{\lebe^1}=|\A(D)(\rho K w)|(\R^n)\qquad\text{as }\varepsilon\downarrow0.
\end{align*}
Furthermore, by \eqref{eq:defs}, we have that
\begin{align*}
\rho Kw-\pi (\rho Kw)=K\star [\A(D)(\rho K w)]=K\star(\delta_0w+\psi),
\end{align*}
and \eqref{eq:conv_legit} follows with $\Psi\coloneqq K\star\psi$.
\end{proof}
As a minor remark, it is well--known that mollifications of bounded measures converge strictly, but we presented an elementary proof in this case, as it was readily available.

Using Lemma~\ref{lem:rhoKw}, which essentially shows that $u=\rho Kw$ is admissible in all estimates investigated, we can give a slightly different proof of \cite[Prop.~5.5,~8.23]{VS}, and \cite[Prop.~3.1]{BVS}, which is contained in the following:
\begin{proof}[Proof of necessity of algebraic assumptions for Theorems~\ref{thm:Lp},~\ref{thm:hardy},~\ref{thm:Linfty}]
We begin by assuming that $\A(D)$ is not canceling and show that the embedding in \ref{it:Lp-lore} fails.

Let $u\coloneqq \rho Kw$ be as in Lemma \ref{lem:rhoKw}, $1\leq j\leq\min\{n-1,k\}$, $1<q<\infty$. By \eqref{eq:conv_legit}, it suffices to show that $D^{k-j}Kw$ is not in $\lebe^{n/(n-j),q}$ in a neigbourhood of zero.

Essentially by Lemma~\ref{lem:canc} (or \cite[Lem.~2.1]{BVS}), we have that $D^{k-j}K$ is $(j-n)$--homogeneous, and so is $D^{k-j}Kw\not\equiv0$. Since $D^{k-j}Kw$ is continuous on $\mathbb{S}^{n-1}$, it follows that there exists a constant $c_0>0$ such that
\begin{align*}
|D^{k-j}K(x)w|\geq c_0|x|^{j-n}\qquad\text{ for all }x\in\mathcal{C},
\end{align*}
where the set $\mathcal{C}=\{x\in\R^n\colon x_1^2+x_2^2+\ldots+x_{n-1}^2<ax_n^2\}$ for some $a>0$ is a cone (represented in a new coordinate system). We can then have, with $B\coloneqq B_1$,
\begin{align*}
\|D^{k-j}Kw\|^{q}_{\lebe^{n/(n-j),q}(B)}&=\int_0^{\infty}\left[\lambda^{\frac{n}{n-j}}\mathscr{L}^n(\{x\in B\colon|D^{k-j}K(x)w|>\lambda\})\right]^{\frac{q(n-j)}{n}}\dfrac{\dif \lambda}{\lambda}\\
&\geq \int_0^{c_0}\left[\lambda^{\frac{n}{n-j}}\mathscr{L}^n(\{x\in B\cap\mathcal{C}\colon |x|<(\lambda/c_0)^{\frac{1}{j-n}}\})\right]^{\frac{q(n-j)}{n}}\dfrac{\dif \lambda}{\lambda}\\
&=\infty,
\end{align*}
since $\mathscr{L}^n({B_r\cap \mathcal{C}})\sim r^n$. This proves necessity of cancellation for Theorem~\ref{thm:Lp}\ref{it:Lp-lore},~\ref{it:Lp-sobo}.

The necessity of cancellation for Theorem~\ref{thm:hardy} follows almost along the same lines, by computing
\begin{align*}
\int_{B}\dfrac{|D^{k-j}K(x)w|^q}{|x|^{n-q(n-j)}}\dif x\geq c_0\int_{B\cap\mathcal{C}}\dfrac{|x|^{q(j-n)}}{|x|^{n-q(n-j)}}\dif x=\infty.
\end{align*}

Necessity of cancellation for Theorem~\ref{thm:Lp}\ref{it:Lp-frac},~\ref{it:Lp-beso},~\ref{it:Lp-trli} follows from the arguments in, respectively, \cite[Prop.~8.19-21]{VS}.

Finally, we prove necessity of Condition~\eqref{eq:mist_cond2} for Theorem~\ref{thm:Linfty}. Let $k\geq n$. Recall from \cite[Sec.~3]{Rdiff} that $\A^\dagger(\cdot)$ extends to a tempered distribution that we do not relabel for which
\begin{align*}
D^{k-n}\mathscr{F}^{-1}[\A^\dagger(\cdot)]\mathbf{w}&=H_0\mathbf{w}+\log|\cdot|\int_{\mathbb{S}^{n-1}}\A^\dagger(\xi)\mathbf{w}\otimes\xi^{\otimes (k-n)}\dif\mathscr{H}^{n-1}(\xi)\\
&\eqqcolon H_0\mathbf{w}+\log|\cdot|\mathcal{L}\mathbf{w}\qquad\text{ for }\mathbf{w}\in W,
\end{align*}
where $H_0\in\hold^\infty(\R^n\setminus\{0\})$ is a tensor--valued zero--homogeneous function and $\mathcal{L}$ is a linear map between finite dimensional spaces. Therefore, we can write $D^{k-n}K=H_0+\log|\cdot|\mathcal{L}$.

We now assume that Condition~\eqref{eq:mist_cond2} fails, so that there exists $w\in W$ such that $w\in\mathrm{im\,}\A(\xi)$ for all $\xi\neq0$ and, in addition, $\mathcal{L}w\neq0$. For small $|x|$ we have
\begin{align*}
|D^{k-n}K(x)w|&\geq \left||H_0(x)w|+\log|x|\mathcal{L}w\right|\\
&\geq -\log|x||\mathcal L w|-\|H_0\|_\infty|w|\rightarrow\infty\quad\text{ as }x\rightarrow0.
\end{align*}
By Lemma~\ref{lem:rhoKw}, it follows that the embedding fails for $u=\rho Kw$.
\end{proof}
\begin{proof}[Proof of sufficiency of algebraic assumptions for Theorems~\ref{thm:Lp},~\ref{thm:hardy},~\ref{thm:Linfty}]
We first prove that the embedding in \ref{it:Lp-beso} holds. This includes sufficiency of cancellation for \ref{it:Lp-frac} (by choosing $q=n/(n-s)$) and for \ref{it:Lp-sobo} (by further choosing $s=j$).

The proof follows as in \cite[Sec.~8]{VS}. We write $f\coloneqq\A(D)u$. Then, by homogeneity of $\A(D)$, we have that
\begin{align*}
[\mathscr{F}(u-\pi u)](\xi)=\A^\dagger(\xi)\hat{f}(\xi)=\A^\dagger\left(\frac{\xi}{|\xi|}\right)|\xi|^{-k}\hat{f}(\xi)=\A^\dagger\left(\frac{\xi}{|\xi|}\right)\widehat{I_k f}(\xi),
\end{align*}
where $I_k$ denotes the $k$-th homogeneous Riesz potential. By \cite[Thm.~5.2.2-3]{Triebel},
\begin{align*}
\|u-\pi u\|_{\dot{\mathrm{B}}{_q^{k-s,n/(n-s)}}}\leq c\|I_kf\|_{\dot{\mathrm{B}}{_q^{k-s,n/(n-s)}}}\leq c\|f\|_{\dot{\mathrm{B}}{_q^{-s,n/(n-s)}}} =c\|\A(D)u\|_{(\dot{\mathrm{B}}{_q^{s,n/s}})^*}.
\end{align*}
One then employs Lemma~\ref{lem:annihilator} and \cite[Prop.~8.8]{VS} with $L(D)=\mathcal{A}(D)$ and $f=\A(D)u$ to conclude. The proof is similar in the case \ref{it:Lp-trli} of Triebel--Lizorkin spaces, using \cite[Prop.~8.7]{VS} instead of \cite[Prop.~8.8]{VS}.

As for Theorem~\ref{thm:Lp}\ref{it:Lp-lore}, we start with the case $j=1$ and use boundedness of singular integrals between Lorentz spaces and \cite[Prop.~8.10]{VS} (as above, for $f=\A(D)u$ and $L(D)=\mathcal{A}(D)$ with $\mathcal{A}(D)$ given by Lemma~\ref{lem:annihilator}) to get
\begin{align*}
\|D^{k-1}(u-\pi u)\|_{\lebe^{n/(n-1),q}}&=\|D^{k-1}K\star \A(D)u\|_{\lebe^{n/(n-j),q}}\leq c\|\A(D)u\|_{\dot{\sobo}{^{-1}\lebe^{n/(n-j),q}}}\\
&=c\|\A(D)u\|_{\left(\dot{\sobo}{^{1}\lebe^{n,q/(q-1)}}\right)^*}\leq c\|\A(D)u\|_{\lebe^1}.
\end{align*}
The case when $j=2,3,\ldots,\min\{n-1,k\}$ follows by an iteration of 
\cite[Prop.~8]{Tartar_lorentz}.

The proof of Theorem~\ref{thm:hardy} is completed by inserting $G=K$ in the proof of \cite[Prop.~2.3]{BVS}, which then requires that \cite[Prop.~2.2]{BVS} holds for constant rank and canceling operators. This is true and can be seen by replacing $L(D)$ with $\mathcal{A}(D)$ as given by Lemma~\ref{lem:annihilator} in the proof of \cite[Prop.~2.2]{BVS}. Similarly, the proof of Theorem~\ref{thm:Linfty} follows from the proof of \cite[Thm.~1.3]{Rdiff}, noting that \cite[Lem.~3.1]{Rdiff} (which is a minor modification of \cite[Lem.~2.2]{BVS}) extends to constant rank operators by Lemma~\ref{lem:annihilator}.
\end{proof}
\subsection{Concluding remarks}
\begin{remark}[Non--constant rank operators]\label{rk:non-CR}
We recall from \cite[Thm.~1]{R} that if \eqref{eq:exact}, holds, then both differential operators $\A(D)$ and $\mathcal{A}(D)$ are of constant rank. In particular, if $\A(D)$ does not have constant--rank, then one cannot argue by duality using estimates for cocanceling operators. Even if one relaxes \eqref{eq:exact} to
\begin{align*}
\ker\mathcal{A}(\xi)\subset\mathrm{im\,}\A(\xi)\qquad\text{ for }\xi\neq0,
\end{align*}
in which case cancellation of $\A(D)$ implies cocancellation of $\mathcal{A}(D)$, one still cannot use the Fourier multiplier techniques to prove the inequalities of this note. This is due to the fact that $0\neq\xi\mapsto\A^\dagger(\xi)$ is necessarily discontinuous. 

In particular, if one dispenses of the assumption that $\A(D)$ is canceling, inequalities for non--eliptic operators may become available, but the methods have to be new. This is the case for \cite[Prop.~5.4]{VS} and the family of pure Hardy--type inequalities in \cite[Sec.~5]{BVS}, which are derived by elementary methods. On the other hand, by Remark~\ref{rk:proj} below, whenever such inequalities hold for non--eliptic operators, these do not have constant rank.
\end{remark}
\begin{remark}[Indispensability of the projection]\label{rk:proj}
If $\A(D)$ is not elliptic but of constant rank, one cannot dispense of the projection $\pi$ defined in \eqref{eq:defs} in any inequality from Section~\ref{sec:statements(for_real)}. In other words, inequalities as in \eqref{eq:VS} cannot hold for constant rank, non--elliptic operators. This so, since by \cite[Thm.~1]{R}, there exists a differential operator $\mathfrak{B}(D)\not\equiv0$ such that $\mathrm{im\,}\mathfrak{B}(\xi)=\ker\A[\xi]\neq 0$ for all $\xi\neq0$. Plugging $u\coloneqq\mathfrak{B}(D)\psi$ for $0\neq\psi\in\hold^\infty_c$ in \eqref{eq:VS} we obtain
\begin{align*}
\|D^{k-1}\mathfrak{B}(D)\psi\|_{\lebe^{n/(n-1)}}\leq c\|\A(D)\circ\mathfrak{B}(D)\psi\|_{\lebe^1}=0,
\end{align*}
which is an obvious contradiction. In particular, if \cite[Prop.~8.1]{VS} would be concluded in the positive, a counter--example should fail the constant rank condition. The same is true if \cite[Open~Prob.~5.1]{VS} would be concluded in the negative.

Thus, the inequalities in Section~\ref{sec:statements(for_real)} can also be intepreted as follows: \emph{If $\A(D)$ has constant rank, among all solutions $u$ of $\A(D)u=f\in\mathcal{M}$, there is one that has the appropriate regularity; moreover the statement is optimal if $\A(D)$ is non--elliptic}. Interestingly, for the operators in \cite[Prop.~5.4]{VS}, \cite[Sec.~5]{BVS}, which necessarily have non--constant rank whenever non--elliptic, one obtains critical regularity for \emph{all} solutions $u$.
\end{remark}
\begin{remark}[Cocanceling estimates from canceling estimates]\label{rk:duality}
Our proof of sufficiency of cancellation for Theorem~\ref{thm:Lp} uses the method from \cite{VS} to derive ``primal'' estimates for canceling operators from ``dual'' estimates for cocanceling operators, by which we mean
\begin{align*}
\|K\star\A u\|_X\leq c\|\A u\|_{Y}\leq c\|\A u\|_{\lebe^1}\qquad\text{ for }u\in\hold^\infty_c(\R^n,V),
\end{align*}
where $K$ is as in \eqref{eq:defs}, $X$ is as in the statement of Theorem~\ref{thm:Lp}, and $Y$ is given by the elliptic estimates used in the proof. In the present context, one can also derive the ``dual'' estimates from ``primal'' estimates. We give an example to illustrate this, namely when $X=\dot{\sobo}{^{k-1,n/(n-1)}}$, so that $Y=({\dot{\sobo}}{^{1,n}})^*$. 

Suppose that we are given operators $\A(D)$ and $\mathcal{A}(D)$ such that the exact relation \eqref{eq:exact} holds (hence both have constant rank by Remark~\ref{rk:non-CR}) such that $\mathcal{A}(D)$ is cocanceling, so $\A(D)$ is canceling. Let $w\in\hold^\infty_c$ be such that $\mathcal{A}(D)w=0$. By \eqref{eq:exact}, we have that $u\coloneqq K\star w\in\mathscr{S}$ is such that $\A[D]u=w$ (see also \cite[Lem.~5]{R}). Assuming that Theorem~\ref{thm:Lp}\ref{it:Lp-sobo} holds, we have that
\begin{align*}
\|w\|_Y&=\|\A(D) u\|_{({\dot{\sobo}}{^{1,n}})^*}\leq c\|D^{k-1}(K\star\A(D)u)\|_{\lebe^{n/(n-1)}}\\
&\leq c\|\A(D)u\|_{\lebe^1}=c\|w\|_{\lebe^1},
\end{align*}
where the second inequality follows by definition chasing and integration by parts.
\end{remark}
We conclude this note with two typical examples of constant rank, non--elliptic operators: $\A(D)=\curl$ and $\A(D)=\di$:
\begin{example}
Let $\A(D)=\curl$ on $\R^n$, $n\geq2$. Then $\A(D)$ is canceling and the estimates in Theorems~\ref{thm:Lp} and \ref{thm:hardy} hold. Moreover, one can use the inequality
\begin{align}\label{eq:curl_est}
\|u-\pi u\|_{\lebe^{n/(n-1)}}\leq c\|\curl u\|_{\lebe^1}\qquad\text{ for }u\in\hold^\infty_c(\R^{n},\R^n)
\end{align}
to prove that
\begin{align}\label{eq:BB_est}
\|u\|_{\lebe^{n/(n-1)}}\leq c\|\curl u\|_{\lebe^1}\quad\text{ for }u\in\hold^\infty_c(\R^{n},\R^n)\text{ with }\di u=0.
\end{align}
In this case, we recover \cite[Thm.~2]{BB04}, which states that \eqref{eq:BB_est} holds for $n=3$, but the approach we give is not new, as the proof below is essentially given in \cite{VS0}.
\end{example}
\begin{proof}[Proof of \eqref{eq:BB_est} using \eqref{eq:curl_est}]
By definition of $\pi$ (see \eqref{eq:defs} and the lines above), we have
\begin{align*}
u=\pi u+(u-\pi u)\eqqcolon \nabla\phi+\psi
\end{align*}
is an $\lebe^2$--Helmholtz decomposition of $u\in\hold^\infty_c(\R^n,\R^n)$ for unique $\phi\in\mathscr{S}(\R^n)$, $\psi\in\mathscr{S}(\R^n,\R^n)$ with $\di \psi=0$. If we assume that $\di u=0$, it follows that $\Delta \phi=0$ (e.g., by \cite[Lem.~2.1]{BVS}), so that $u-\pi u=u$. The proof is complete.
\end{proof}
\begin{example}
Let $\A(D)=\di$ on $\R^n$, $n\geq1$. Then $\A(D)$ is not canceling and the estimates in Theorems~\ref{thm:Lp} and \ref{thm:hardy} fail (cp.~\cite[Rk.~5]{BB07}). If $n=1$, Condition~\eqref{eq:mist_cond2} trivially holds, so the estimate of Theorem~\ref{thm:Linfty} holds, but this is not very interesting as the inequality follows by the Fundamental Theorem of Calculus anyway.
\end{example}
\section{Appendix: Weak--type estimates}
We briefly discuss the remaining limiting cases $q\in\{0,\infty\}$ of Theorem~\ref{thm:Lp}\ref{it:Lp-beso}, \ref{it:Lp-lore}. If $q=1$, the problem is open (cp. \cite[Open Prob.~8.2-3]{VS}). If $q=\infty$, one obtains, by analogy with \cite[Prop.~8.22,~8.24]{VS} and \cite[Prop.~7.5]{Rdiff}:
\begin{proposition}
Let $\A(D)$ be a constant rank operator as in \eqref{eq:A}, $n\geq 1$, $s\in (0,n)$, and $1\leq j\leq\min\{n,k\}$. We have that
\begin{align*}
\|D^{k-1}(u-\pi u)\|_{\lebe^{n/(n-j),\infty}}&\leq c\|\A u\|_{\lebe^1}\\
\|u-\pi u\|_{\dot{\mathrm{B}}{_\infty^{k-s,n/(n-s)}}}&\leq c\|\A u\|_{\lebe^1}
\end{align*}
for all $u\in\hold^\infty_c(\R^n,V)$.
\end{proposition}
If $j=n$, we write $\lebe^{\infty,\infty}$ for the space weak--$\lebe^\infty$, as defined in \cite{BDVS}.

\end{document}